\documentclass[a4paper,12pt]{article}
\usepackage{amsmath,amsthm,amssymb,amsfonts}
\usepackage{bbm,bm}
\usepackage{latexsym}
\usepackage{mathrsfs}
\usepackage{threeparttable}
\usepackage{tabularx}
\usepackage{booktabs}
\usepackage{graphicx,subfigure}
\usepackage{color}
\usepackage{indentfirst}
\usepackage{geometry}
\usepackage{caption}
\usepackage{lineno}
\usepackage{abstract}
\usepackage{enumerate,mdwlist}
\usepackage[numbers,sort&compress]{natbib}
\usepackage[colorlinks,linkcolor=blue,citecolor=red]{hyperref}
\usepackage{epstopdf}
\usepackage{pifont}

\def \[{\begin{equation}}
\def \]{\end{equation}}

\newtheorem{thm}{Theorem}[section]
\newtheorem{prop}[thm]{Proposition}
\newtheorem{defi}[thm]{Definition}
\newtheorem{claim}{Claim}
\newtheorem{case}{Case}
\newtheorem{lem}[thm]{Lemma}
\newtheorem{cor}[thm]{Corollary}

\newtheorem{conj}[thm]{Conjecture}

\begin{document}

\setlength{\baselineskip}{20pt}
\begin{center}

{\Large \bf The minimum degree of minimal $k$-factor-critical claw-free graphs$^{\text{\ding{73}}}$}

\vspace{4mm}

{Jing Guo$^{\rm 1}$, Qiuli Li$^{\rm 1}$, Fuliang Lu$^{\rm 2}$, Heping Zhang$^{\rm 1 \ast}$}

\vspace{4mm}

\footnotesize{$^{\rm 1}$School of Mathematics and Statistics, Lanzhou University,
Lanzhou, 730000, PR China}

\footnotesize{$^{\rm 2}$School of Mathematics and Statistics, Minnan Normal University,
Zhangzhou, 363000, PR China}
\renewcommand\thefootnote{}
\footnote{$^{\text{\ding{73}}}$ This work is supported by NSFC\,(Grant No. 12271229 and 12271235) and
NSF of Fujian Province (Grant No. 2021J06029).}

\footnote{$^{\ast}$ The corresponding author. \\E-mail addresses:
guoj20@lzu.edu.cn (J. Guo), qlli@lzu.edu.cn (Q. Li), flianglu@163.com (F. Lu)
and zhanghp@lzu.edu.cn (H. Zhang).}

\end{center}

\noindent {\bf Abstract}:
A graph $G$ of order $n$ is said to be $k$-factor-critical for integers $1 \leq k< n$,
if the removal of any $k$ vertices results in a graph with a perfect matching.
A $k$-factor-critical graph is minimal if the deletion of any edge results in a graph
that is not $k$-factor-critical. In 1998, O. Favaron and M. Shi conjectured that
every minimal $k$-factor-critical graph has minimum degree $k+1$.
In this paper, we confirm the conjecture for minimal $k$-factor-critical claw-free graphs.
Moreover, we show that every minimal $k$-factor-critical claw-free graph $G$
has at least $\frac{k-1}{2k}|V(G)|$ vertices of degree $k+1$
in the case that $G$ is $(k+1)$-connected, yielding a further evidence for
S. Norine and R. Thomas' conjecture on the number of vertices of minimum degree of minimal bricks
for $k=2$.

\vspace{2mm}
\vspace{2mm}


 {\setcounter{section}{0}
\section{Introduction}\setcounter{equation}{0}

All graphs considered in this paper are finite and simple.
Let $G$ be a graph with vertex set $V(G)$ and edge set $E(G)$.
The {\em order} of a graph $G$ is the cardinality of $V(G)$.
A {\em perfect matching} $M$ of $G$ is a set of edges
such that each vertex is incident with exactly one edge of $M$.
A {\em factor-critical} graph is a graph in which
the removal of any vertex results in a graph with a perfect matching.
A graph $G$ with at least one edge is called {\em bicritical} if,
after the removal of any pair of distinct vertices of $G$,
the resulting graph has a perfect matching.
The concepts of factor-critical and bicritical graphs were introduced
by T. Gallai \cite{GT} and L. Lov\'{a}sz \cite{LL}, respectively.
In the matching theory, factor-critical graphs and bicritical graphs are
two basic building blocks in decompositions of graphs.

O. Favaron \cite{F} and Q. Yu \cite{Y} independently introduced
$k$-factor-critical graphs as a generalization of factor-critical and
bicritical graphs. A graph $G$ of order $n$ is said to
be {\em $k$-factor-critical} for integers $1 \leq k< n$,
if the removal of any $k$ vertices results in a graph with a perfect matching.
They also gave the following basic property on connectivity
of $k$-factor-critical graphs.

\begin{thm}[\cite{F, Y}]\label{conn}
If $G$ is a $k$-factor-critical graph of order $n$,
then $G$ is $k$-connected, $(k+1)$-edge-connected,
and $(k-2)$-factor-critical whenever $k \geq 2$.
\end{thm}

For further results concerning $k$-factor-critical graphs,
 the reader may refer to \cite{PS, PMD, N, LP, LDYQ, YL}.
Given a vertex $u$ of $G$, a vertex $v$ is called a {\em neighbor} of $u$
if $uv\in E(G)$. Let $N(u)$ denote the {\em neighborhood},
the set of neighbors of vertex $u$ in $G$,
and $d_{G}(u)$ the {\em degree} of $u$ in $G$,
the cardinality of $N(u)$.
The {\em minimum degree} of $G$, denoted by $\delta(G)$,
is the minimum value in degrees of all vertices of $G$.

Theorem \ref{conn} indicates that every $k$-factor-critical graph
has minimum degree at least $k+1$. O. Favaron and M. Shi \cite{FS}
focused on minimal $k$-factor-critical graphs.
A $k$-factor-critical graph $G$ is {\em minimal} if $G-e$ is
not $k$-factor-critical for each  edge $e$ of $G$.
They obtained that every minimal $k$-factor-critical graph of
order $n$ has minimum degree $k+1$ for $k=1, n-6, n-4$ and $n-2$, and
 put forward a general problem:  is it true  that  every minimal
$k$-factor-critical graph of order $n$ has minimum degree $k+1$?
In \cite{ZWL}, Z. Zhang et al. redescribed the problem as the following conjecture.

\begin{conj}[\cite{FS, ZWL}]\label{conj1}
Let $G$ be a minimal $k$-factor-critical graph of order $n$ with $1 \leq k<n$.
Then $\delta(G)=k+1$.
\end{conj}

Recently, a novel method was used to confirm that
Conjecture \ref{conj1} is true for $k=n-8$ and $n-10$
(see \cite{GZ1, GZ2}). For $k=2$, J. Guo et al. \cite{GZ3}
showed that every minimal bicritical graph with four or more vertices
has at least four vertices of degree three by using a
brick decomposition on bicritical
graphs with a 2-separation (see \cite{ZWY22}) and recent results on minimal bricks.
Q. Li, F. Lu and H. Zhang \cite{LLZH} confirmed Conjecture \ref{conj1} for planar graphs.

\begin{thm}[\cite{LLZH}]\label{main01}
Let $G$ be a minimal $k$-factor-critical planar graph. Then $\delta(G)=k+1$.
\end{thm}

A {\em brick} is a 3-connected bicritical graph \cite{ELW},
which plays a key role in matching theory of graphs.
J. Edmonds et al. \cite{ELW} and L. Lov\'{a}sz \cite{LO} proposed and developed
the ``tight cut decomposition" of matching covered graphs into
list of bricks and braces in an essentially unique manner.
This decomposition reduces several problems from matching theory to bricks
(e.g. a graph is Pfaffian if and only if its
bricks and braces are Pfaffian \cite{LR}).

A brick $G$ is {\em minimal} if $G-e$ is not a brick
for every edge $e$ of $G$.
One may easily deduce that
a 3-connected minimal bicritical graph $G$ is also a minimal brick
since, for any edge $e$ of $G$,
$G-e$ is not bicritical, yielding $G-e$ is not a brick.
But the converse does not always hold. For example, Fig. \ref{tu-01}
presents a minimal brick $G$ rather than
a minimal bicritical graph as $G-e$ is also bicritical.

\begin{figure}[h]
\centering
\includegraphics[height=2.9cm,width=5.3cm]{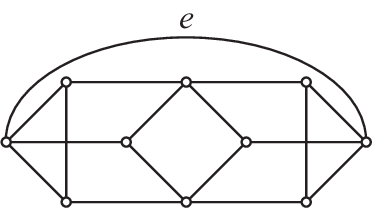}
\caption{\label{tu-01} A minimal brick $G$ rather than
a minimal bicritical graph.}
\end{figure}

There has been much research on the minimum degrees of minimal bricks.
M. H. de Carvalho et al. \cite{CLM} proved that
every minimal brick contains a vertex of degree three.
S. Norine and R. Thomas \cite{NT} presented a recursive procedure
for generating minimal bricks and showed that every minimal brick has
at least three vertices of degree three, and then went on to pose
the following conjecture.

\begin{conj}[\cite{NT}]\label{conj2}
There exists $\alpha>0$ such that every minimal brick $G$
has at least $\alpha|V(G)|$ vertices of degree three.
\end{conj}

Furthermore, F. Lin et al. \cite{LZL} showed the existence of four such vertices.
X. He and F. Lu \cite{HL} proved that every solid minimal brick $G$
has at least $\frac{2}{5}|V(G)|$ vertices of degree three,
which partly supports Conjecture \ref{conj2}.

A graph is said to be {\em claw-free} if it does not contain any
induced subgraph isomorphic to the complete bipartite graph $K_{1,3}$.
Claw-free graphs have been a popular condition to
study for many graphic parameters.
In this paper, we confirm Conjecture \ref{conj1} for all minimal
$k$-factor-critical claw-free graphs. Moreover,
we show that if the connectivity of $G$ is increased
(to $(k+1)$-connected), then $G$ has at least $\frac{k-1}{2k}|V(G)|$ vertices of degree $k+1$.
The following theorems are our main results.

\begin{thm}\label{main1}
Let $G$ be a minimal $k$-factor-critical claw-free graph. Then $\delta(G)=k+1$.
\end{thm}

\begin{thm}\label{main2}
Let $G$ be a minimal $k$-factor-critical claw-free graph and $k\geq 2$.
If $G$ is $(k+1)$-connected, then $G$ has at least
$\frac{k-1}{2k}|V(G)|$ vertices of degree $k+1$.
\end{thm}

Specially, if $k=2$, we derive that every 3-connected minimal bicritical
claw-free graph $G$ has at least $\frac{1}{4}|V(G)|$ vertices of degree
three by Theorem \ref{main2}, which yields further evidence for Conjecture \ref{conj2}.

For a minimal $k$-connected graph $G$, there exist similar results.
W. Mader \cite{MD} showed that $\delta(G)=k$ and \cite {MW} further
obtained that $G$ has at least $\frac{(k-1)|V(G)|+2}{2k-1}$ vertices of degree $k$.

Some preliminaries are presented in Section 2.
The proofs of Theorems \ref{main1} and \ref{main2} will be given in Sections 3 and 4, respectively.

\section{Some preliminaries}

In this section, we give some graph-theoretical terminology and notation,
and some known results that will be used in the proofs of the main results.
For unspecific graph-theoretic terminology, we follow
J. A. Bondy and U. S. R. Murty \cite{BM}.

For $X, Y \subseteq V(G)$, by $E(X, Y)$ we mean the set of edges
of $G$ with one end vertex in $X$ and the other end vertex in $Y$.
Denote by $G[S]$ the subgraph of $G$ induced
by $S$ for a subset $S$ of $V(G)$, and $G-S=G[V(G)\backslash S]$.
If $F \subseteq E(G)$, then the {\em edge-induced subgraph}
$G[F]$ is the subgraph of $G$ whose edge set is $F$ and whose
vertex set consists of the end-vertices of all edges of $F$.
$G-uv$ stands for the graph obtained from $G$ by deleting $uv$ if $uv \in E(G)$.
Likewise, $G+uv$ stands for the graph obtained from $G$ by adding an
edge $uv$ if $uv \notin E(G)$.
A {\em complete graph} is a graph
in which any two vertices are adjacent.

Let $G$ be a connected graph and let $k$ be a positive integer
with $k<|V(G)|$. A {\em $k$-vertex cut} of $G$ is
a set of vertices of $G$ with cardinality $k$ whose removal disconnects $G$,
and $G$ is {\em $k$-connected}
if $G$ contains no vertex cut with less than $k$ vertices.
Similarly, a graph $G$ is said to be {\em $k$-edge-connected}
if the deletion of less than $k$ edges from $G$ does not disconnect it.
We call a graph $G$ {\em trivial} if $|V(G)|=1$,
and {\em nontrivial} otherwise.


W. T. Tutte \cite{TTW} established the following fundamental theorem.

\begin{thm}[Tutte's theorem]\label{Tutte}
A graph $G$ has a perfect matching if and only if
\begin{center}
$c_{o}(G-X) \leq |X|$
\end{center}
for every subset $X$ of $V(G)$,
where $c_{o}(G-X)$ denotes the number of odd components of $G-X$.
\end{thm}

The following characterization of minimal $k$-factor-critical graphs
of Tutte's type due to O. Favaron and M. Shi \cite{FS} will be useful
in the proof of Theorem \ref{main1}.

\begin{thm}[\cite{FS}]\label{minimal}
Let $G$ be a $k$-factor-critical graph. Then $G$ is minimal
if and only if for each edge $e=uv \in E(G)$,
there exists $S_{e} \subseteq V(G)\backslash \{u,v\}$
such that $|S_{e}| \geq k$, $c_{o}(G_{e})=|S_{e}|-k+2$ and $u$ and $v$
belong to two distinct odd components of $G_{e}$,
where $G_{e}=G-e-S_{e}$.
\end{thm}

O. Favaron and M. Shi \cite{FS} found that from the ear decomposition of
a factor-critical graph, the last ear is a path or cycle of length at least 3
and then every minimal factor-critical graph has minimum degree two.
We state it as a lemma for later use.

\begin{lem}[\cite{FS}]\label{lem}
Let $G$ be a minimal factor-critical graph. Then $\delta(G)=2$.
\end{lem}

\section{Proof of Theorem \ref{main1}}

In this section, we prove that Conjecture \ref{conj1} holds for all minimal
$k$-factor-critical claw-free graphs.
We first prove some useful lemmas by Theorem \ref{minimal}.

\begin{lem}\label{lem0}
Let $G$ be a minimal $k$-factor-critical claw-free graph.
For every edge $e$ of $G$, let $S_e$ and $G_e$ be defined in Theorem $\ref{minimal}$.
Then for any subset $X$ of $S_e$ with $|X|=k$, every perfect matching
of $G-X$ contains $e$. Moreover, if $|S_e|\geq k+1$, then every vertex in $S_e$
has a neighbor in some odd component of $G_e$
other than the odd components containing $u$ and $v$.
\end{lem}

\begin{proof}
By Theorem \ref{minimal}, $c_{o}(G_{e})=|S_{e}|-k+2$.
Since $G$ is $k$-factor-critical, $G-X$ has a perfect matching $M$.
If $|S_{e}|=k$, then $c_{o}(G_{e})=2$
and $e$ joins the two odd components. So $e \in M$.
If $|S_{e}| \geq k+1$, then $c_{o}(G_{e}) \geq 3$.
For any vertex $w$ in $S_e$, let $X\subseteq S_e\backslash \{w\}$.
Since $|S_e\backslash X|=c_{o}(G_{e})-2$, $e \in M$ and
$w$ is matched to a vertex in some odd component of $G_e$
other than the odd components containing $u$ and $v$ by $M$.
\end{proof}

\begin{lem}\label{lem1}
Let $G$ be a minimal $k$-factor-critical claw-free graph with $k \geq 2$.
For any edge $e=uv \in E(G)$ with $d_{G}(u) \geq k+2$ and $d_{G}(v) \geq k+2$,
let $S_e$ be a smallest set satisfying the conclusion of Theorem $\ref{minimal}$
when applied to the edge $e$ and $G_{e}=G-e-S_{e}$.
Let $O_1, \ldots, O_t$ be the odd components of $G_e$ such that $u \in V(O_1)$ and $v \in V(O_2)$.
Then $G_{e}$ has no even component and one of the following statements holds:

{\rm (\romannumeral1)} $|S_e|=k$, $t=2$ and $|V(O_i)|>1$ for $i=1,2$;

{\rm (\romannumeral2)} $|S_e|=k+1$, $t=3$ and $|V(O_i)|=1$ for $i=1,2$; moreover,
for every vertex $w$ of $S_e$, $\{uw,vw\}\subset E(G_{e})$ and $N(w)\cap V(O_3)\neq \emptyset$.
\end{lem}

\begin{proof}
By Theorem \ref{minimal}, $t=|S_e|-k+2$.
We first prove the following claim.

\smallskip
\smallskip
\noindent \textbf{Claim 1.}
Each $O_i$ is connected to at least $k+1$ vertices
in $V(G) \backslash V(O_i)$, $i=1, 2, \ldots, t$.
\smallskip
\smallskip

Since $G$ is $k$-factor-critical, $G$ is $k$-connected by Theorem \ref{conn}.
Then each $O_i$ is connected to at least $k$ vertices outside $O_i$.
Note that each $O_i$ is an odd component of $G_e$.
If there exists an $O_i$ such that it is connected to exactly
$k$ vertices outside it, then the removal of these $k$ vertices
isolates the odd component $O_i$.
Consequently, the resulting graph has no perfect matching,
contradicting that $G$ is $k$-factor-critical.
So Claim 1 holds.

\bigskip
By Theorem \ref{minimal}, we know that $|S_e|\geq k$. We will prove that
if $|S_e|=k$, then {\rm (\romannumeral1)} holds and if $|S_e|\geq k+1$,
then $|S_e|=k+1$ and {\rm (\romannumeral2)} holds.

\begin{case}\label{case1}
$|S_{e}|=k$.
\end{case}

Under this conndition, $G_{e}$ has exactly two odd components $O_1$ and $O_2$.
By Claim 1, every $O_{i}$ is connected to all vertices of $S_{e}$.
So every vertex in $S_{e}$ has neighbors in both $O_1$ and $O_2$.
If $O_1$ (resp. $O_2$) is trivial, then $N(u)=S_e \cup \{v\}$
(resp. $N(v)=S_e \cup \{u\}$). Thus $d_G(u)=k+1$ (resp. $d_G(v)=k+1$).
This is a contradiction to the assumption that $d_{G}(u) \geq k+2$
(resp. $d_{G}(v) \geq k+2$).
Hence both $O_1$ and $O_2$ are nontrivial.
By symmetry, we only consider $O_1$. We claim that there exists a vertex $x$
in $S_e$ such that $N(x) \cap (V(O_1) \backslash \{u\})\neq \emptyset$.
Otherwise, $\{u\}$ is a vertex cut of $G$,
contradicting that $G$ is $k$-connected ($k \geq 2$).
Let $x_1\in N(x) \cap (V(O_1) \backslash \{u\})$ and $x_2\in N(x) \cap V(O_2)$.

Finally, we prove that $G_e$ has no even component.
Suppose to the contrary that $G_e$ has an even component $C$.
Since $G$ is $k$-connected and $C$ can only be connected to $S_e$,
$C$ is connected to all vertices of $S_e$.
Then $N(x) \cap V(C)\neq \emptyset$. Let $x_3\in N(x) \cap V(C)$.
Thus $G[\{x, x_1, x_2, x_3\}]$ is a claw, a contradiction.
So {\rm (\romannumeral1)} holds.

\begin{case}\label{case2}
$|S_{e}| \geq k+1$.
\end{case}

We first show that
for every vertex $w$ of $S_e$, exactly three
odd components of $G_e$ contain vertices in $N(w)$ and $\{uw,vw\}\subset E(G_{e})$.
\hfill $(\dag)$

Let $X \subseteq S_e$ and $|X|=k$. Take any vertex $w$ in $S_e\backslash X$.
By Lemma \ref{lem0}, every perfect matching of $G-X$ contains $e$ and
$N(w) \cap V(O_i)\neq \emptyset$ for some $i \geq 3$.

Suppose $w$ has neighbors in at most two odd components of $G_e$.
Then $w$ has neighbors in at most one of $O_1$ and $O_2$.
Let $S_e'=S_e \backslash \{w\}$. Adding $w$ back to the graph $G_e$ either
turns an odd component into an even component (if $w$ has neighbors in
exactly one odd component) or merges two odd components into a single
odd component (if $w$ has neighbors in two odd components).
Thus $c_o(G-e-S_e')=c_o(G-e-S_e)-1=|S_e|-k+2-1=|S_e'|-k+2$
and $u$ and $v$ are in distinct odd components of $G-e-S_e'$.
Then for the edge $e$, $S_e'$ and $G_e'=G-e-S_e'$
satisfy the conclusion of Theorem \ref{minimal}.
But $|S_e'|=|S_e|-1$, a contradiction to the minimality of $S_e$.

Suppose $w$ has neighbors in at least four odd components of $G_e$.
Since the only edge between different odd components of $G_e$ is $e$,
$w$ has neighbors in at least three components of $G_e+e$.
Let three of them be $\{w_1, w_2, w_3\}$.
Then $G[\{w, w_1, w_2, w_3\}]$ is a claw, a contradiction.
Finally, if $w$ has neighbors in exactly three
odd components but $uw\notin E(G)$ or $vw\notin E(G)$, then $G$ also has a claw,
a contradiction. By the above arguments, $w$ satisfies $(\dag)$.

Now we prove that $G_e$ has no even component.
Suppose to the contrary that $G_e$ has an even component $C$.
Since $G$ is $k$-connected and $C$ can only be connected to $S_e$,
$C$ is connected to every vertex in $S_e$.
Then there exists a vertex $x$ in $S_e$ such that $N(x) \cap V(C)\neq  \emptyset$.
Note that $x$ satisfies $(\dag)$.
Thus $x$ has neighbors in four components of $G_e$ and
the only edge between different components of $G_e$ is $e$.
So $G$ has a claw, a contradiction.

In the sequel, we will show that $G_e$ has exactly three odd components.

We construct a graph $G'$ by contracting every $O_i$ into a
single vertex and deleting all the multiple edges appeared in this process.
Let $U_{e}$ denote the set of the vertices
resulting from the contraction of all $O_i's$.
Then $|U_{e}|=t=|S_{e}|-k+2$ and $V(G')=S_{e}\cup U_{e}$.
For convenience, we use $u'$ and $v'$ to represent
the vertices resulting from the contraction of
$O_1$ and $O_2$, respectively.
Then $u'v' \in E(G')$. By Claim 1,
each $O_i$ is connected to at least $k+1$ vertices in
$V(G) \backslash V(O_i)$, $i=1, 2, \ldots, t$.
It follows that every vertex in $U_{e} \backslash \{u', v'\}$ has
at least $k+1$ neighbors in $S_{e}$.
Thus there are at least $(k+1)(|U_{e}|-2)$
edges from $U_{e} \backslash \{u', v'\}$ to $S_{e}$ in $G'$.

Conversely, since every vertex in $S_{e}$ satisfies $(\dag)$,
it must have a neighbor in exactly one $O_{i}$ $(i \geq 3)$.
Then there are exactly $|S_{e}|$ edges
from $S_{e}$ to $U_{e} \backslash \{u', v'\}$ in $G'$.
So we have the following inequality
\begin{center}
$|S_{e}| \geq (k+1)(|U_{e}|-2)=(k+1)(|S_{e}|-k)$,
\end{center}
which implies that $|S_{e}| \leq k+1$.
By the assumption that $|S_{e}| \geq k+1$, we have $|S_{e}|=k+1$.
It follows that $t=c_o(G_{e})=3$.
Up to now, we have proved that $G_e$ has exactly three odd components
and no even component.
For every vertex $w$ in $S_{e}$, $\{uw, vw\} \subset E(G_e)$
and $N(w) \cap V(O_3)\neq \emptyset$.

We are left to show that both $O_1$ and $O_2$ are trivial.
Without loss of generality, we assume that $O_{1}$ is nontrivial.
Suppose that there exists a vertex $x$ in $S_e$
such that $N(x) \cap (V(O_{1}) \backslash \{u\})\neq \emptyset$.
Since $x$ satisfies $(\dag)$, $vx\in E(G)$ and $N(x) \cap V(O_3)\neq \emptyset$.
Let $x_1\in N(x) \cap V(O_{1}) \backslash \{u\}$ and $x_2\in N(x) \cap V(O_3)$.
Then $G[\{x, x_1, x_2, v\}]$ is a claw, a contradiction.
So none of the vertices in $S_{e}$ have a neighbor in $V(O_{1}) \backslash \{u\}$.
But then $\{u\}$ is a vertex cut of $G$,
contradicting that $G$ is $k$-connected $(k \geq 2)$.
Similarly, $O_{2}$ is trivial.
So {\rm (\romannumeral2)} holds.
\end{proof}

By Lemma \ref{lem1}, we see that
if $S_e$ is a smallest set satisfying the conclusion of Theorem \ref{minimal}
when applied to edge $e$,
then $S_e$ satisfies the conclusions of Lemma \ref{lem1}.
In subsequent discussions, we bear in mind that $S_e$ is always a certain smallest set
whenever applied Theorem \ref{minimal} to the edge $e$ and $G_{e}=G-e-S_{e}$.
Let $O_1, O_2$, $O_3$ denote the odd components of $G_e$
such that $u \in V(O_1)$ and $v \in V(O_2)$.

\begin{defi}\label{def1}
Let $G$ be a minimal $k$-factor-critical claw-free graph with $k\geq 2$.
For any edge $e=uv\in E(G)$ with $d_{G}(u) \geq k+2$ and $d_{G}(v) \geq k+2$,
let $S_e$ be a smallest set satisfying the conclusion of Theorem $\ref{minimal}$
when applied to the edge $e$. Let $G_{e}=G-e-S_{e}$.
We say $e$ is of type $1$ if $G_e$ satisfies Statement {\rm (\romannumeral1)}
of Lemma $\ref{lem1}$; otherwise, $e$ is of type $2$.
\end{defi}

\begin{figure}[h]
\centering
\includegraphics[height=4cm,width=8.6cm]{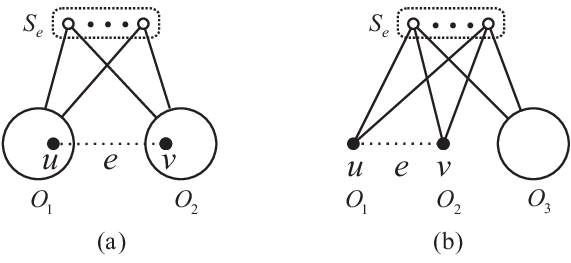}
\caption{\label{tu-02} (a) $e$ is of type 1 if $|S_e|=k$; (b) $e$ is of type 2 if $|S_e|=k+1$.}
\end{figure}
If $e$ is of type $1$,
then $G$ has the configuration as shown in Fig. \ref{tu-02}\ (a);
if $e$ is of type $2$, then
 $G$ has the configuration as shown in Fig. \ref{tu-02}\ (b)
by Lemma \ref{lem1}. Moreover, we have the following proposition.

\begin{prop}\label{prop1}
Let $G$ be a minimal $k$-factor-critical claw-free graph with $k\geq 2$.
For any edge $e=uv\in E(G)$ with $d_{G}(u) \geq k+2$ and $d_{G}(v) \geq k+2$,
$e$ is either of type $1$ or of type $2$, not both.
\end{prop}

Note that if $e$ is of type 2, then
$N(u)=S_{e} \cup \{v\}$ and $N(v)=S_{e} \cup \{u\}$.
It follows that $N(u) \backslash \{v\}=N(v)\backslash\{u\}=S_{e}$.
We state it as a proposition.

\begin{prop}\label{prop0}
Let $G$ be a minimal $k$-factor-critical claw-free graph with $k\geq 2$.
Let $uv$ be an edge of $G$. If $uv$ is of type $2$,
then $N(u) \backslash \{v\}=N(v) \backslash\{u\}$.
\end{prop}

\begin{prop}\label{prop2}
Let $G$ be a minimal $k$-factor-critical claw-free graph with $k\geq 2$.
For any vertex $u$ of $G$, $u$ is incident with at most one edge of type $2$.
\end{prop}

\begin{proof}
Let $e=uv$ and $e$ be of type 2. By Lemma \ref{lem1}, every vertex in $S_{e}$ has
neighbors in the three odd components of $G_{e}$.
Then for any $u_{1}\in N(u)\backslash\{v\}=S_{e}$, $N(u_{1}) \cap V(O_3)\neq \emptyset$.
Thus $N(u)\backslash\{u_{1}\} \neq N(u_{1})\backslash\{u\}$ as $N(u) \cap V(O_3)=\emptyset$.
By Proposition \ref{prop0}, $uu_{1}$ is not of type 2.
Noticing that $u_{1}$ is an arbitrary vertex in $N(u) \backslash \{v\}$,
the result follows.
\end{proof}

\begin{defi}\label{def2}
Let $e=uv \in E(G)$. We say a vertex cut $X_{e}^{u} \subseteq V(G)$
$($resp. $X_{e}^{v} \subseteq V(G)$$)$ satisfies Property $Q$ if

{\rm (\romannumeral1)} $|X_{e}^{u}|=k+1$ $($resp. $|X_{e}^{v}|=k+1$$)$;

{\rm (\romannumeral2)} $v\in X_{e}^{u}$ $($resp. $u\in X_{e}^{v}$$)$;

{\rm (\romannumeral3)} $u$ lies in an odd component, say $C_u$, of
$G-X_{e}^{u}$ $($resp. $v$ lies in an odd component, say $C_v$, of $G-X_{e}^{v}$$)$; and

{\rm (\romannumeral4)} $E(V(C_u), \{v\})=\{uv\}$
$($resp. $E(V(C_v), \{u\})=\{uv\}$$)$.
\end{defi}

Let $G$ be a minimal $k$-factor-critical claw-free graph with $k \geq 2$.
For any edge $e=uv\in E(G)$, if $e$ is of type 1, then $S_e\cup \{v\}$ is a
$(k+1)$-vertex cut and $u$ lies in the odd component $O_1$ of $G-(S_e\cup \{v\})$.
Moreover, $E(V(O_1), \{v\})=\{uv\}$. These observations lead to the following lemma.

\begin{lem}\label{lem00}
Let $G$ be a minimal $k$-factor-critical claw-free graph with $k\geq 2$.
For an edge $uv$ of $G$, if $uv$ is of type $1$,
then there exists a $(k+1)$-vertex cut $S_{uv}\cup \{v\}$
$($resp. $S_{uv}\cup \{u\}$$)$ satisfying Property $Q$.
\end{lem}

We are now ready to prove Theorem \ref{main1}.

\bigskip
\noindent{\bf Proof of Theorem \ref{main1}.}
By Lemma \ref{lem}, the result holds for $k=1$.
Thus we assume that $k\geq 2$ in the following.

Suppose to the contrary that $\delta(G) \geq k+2 \geq 4$.
By Proposition \ref{prop1}, for any edge $e$ of $G$, $e$ is either of
type 1 or of type 2.
Combining with Proposition \ref{prop2}, every vertex of $G$ is incident with
at least three edges of type 1.

By Lemma \ref{lem00}, for every edge of type 1 of $G$,
there exists a $(k+1)$-vertex cut satisfying Property $Q$.
Among all $(k+1)$-vertex cuts satisfying Property $Q$ of $G$,
we choose a $(k+1)$-vertex cut $X_{e}^{u}$ such that
the odd component of $G-X_{e}^{u}$ containing $u$ is of the minimum order,
where $e=uv\in E(G)$.
Let $G_{1}$ be the odd component of $G-X_{e}^{u}$ containing $u$ and
$G_{2}=G-X_{e}^{u}-V(G_{1})$ (as shown in Fig. \ref{tu-03} (a)).
Since $E(V(G_1), \{v\})=\{uv\}$, every perfect matching
of $G-(X_{e}^{u}\backslash \{v\})$ contains $uv$. Then $|V(G_{2})|$ is even.

\begin{figure}[h]
\centering
\includegraphics[height=4.5cm,width=8.6cm]{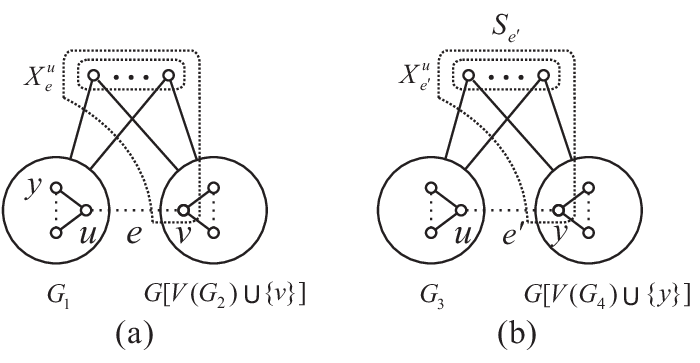}
\caption{\label{tu-03} (a) $X_{e}^{u}$ satisfies Property $Q$;
      (b) $X_{e'}^{u}$ satisfies Property $Q$.}
\end{figure}

Since $d_{G}(u)\geq k+2$, $G_{1}$ is nontrivial.
Then $|N(u) \cap V(G_{1})|\geq 1$. Let $y\in N(u) \cap V(G_{1})$ and $e'=uy$.
Because $v\in N(u)$ and $v\notin N(y)$,
$N(u) \backslash \{y\} \neq N(y)\backslash \{u\}$.
By Propositions \ref{prop1} and \ref{prop0}, $e'$ is of type 1.
Then there exists $S_{e'} \subseteq V(G)\backslash \{u, y\}$ with $|S_{e'}|=k$
such that $G-e'-S_{e'}$ has exactly two odd components containing
$u$ and $y$, respectively.
Let $X_{e'}^{u}=S_{e'} \cup \{y\}$.
By Lemma \ref{lem00}, $X_{e'}^{u}$ satisfies Property $Q$.
Let $G_{3}$ be the odd component of $G-X_{e'}^{u}$ containing $u$ and
$G_{4}=G-X_{e'}^{u}-V(G_{3})$ (as shown in Fig. \ref{tu-03} (b)).
Since $G[V(G_4)\cup \{y\}]$ is an odd component of $G-e'-S_{e'}$,
$|V(G_{4})|$ is even.

For convenience, we present some notation below.

Let $T=X_{e}^{u} \cap X_{e'}^{u}$, $X'=X_{e}^{u} \cap V(G_{3})$, $X''=X_{e}^{u} \cap V(G_{4})$,
$Y'=X_{e'}^{u} \cap V(G_{1})$ and $Y''=X_{e'}^{u}\cap V(G_{2})$.
Then $X_{e}^{u}=X' \cup X'' \cup T$ and $X_{e'}^{u}=Y' \cup Y'' \cup T$.
Let $I_{0}=V(G_{1})$ $\cap$ $V(G_{3})$, $I_{1}=V(G_{1}) \cap V(G_{4})$,
$I_{2}=V(G_{2}) \cap V(G_{4})$ and $I_{3}=V(G_{2}) \cap V(G_{3})$
(as shown in Fig. \ref{tu-04}). So $u \in V(G_{1}) \cap V(G_{3})=I_{0}$.
Since $y \in V(G_{1})$ and $y\in X_{e'}^{u}$, $y\in X_{e'}^{u} \cap V(G_{1})=Y'$.
Moreover, $v \in X' \cup T$ as $v \in N(u)$.

\begin{figure}[h]
\centering
\includegraphics[height=5cm,width=5cm]{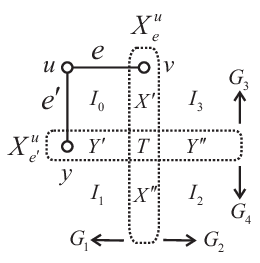}
\caption{\label{tu-04} $X_{e}^{u}$ and $X_{e'}^{u}$ separate the graph $G$.}
\end{figure}

If $N(u) \cap I_{0}\neq \emptyset$, then let $u_0 \in N(u) \cap I_{0}$.
Since $N(v) \cap V(G_{1})=\{u\}$ and $N(y)$ $\cap$ $V(G_{3})=\{u\}$,
$\{vy, u_0v, u_0y\} \cap E(G)=\emptyset$.
Then $G[\{u, v, y, u_0\}]$ is a claw, a contradiction.
So $N(u) \cap I_{0}=\emptyset$. Thus

\begin{equation}\label{30}
N(u) \subseteq X' \cup Y' \cup T.
\end{equation}

Since $\delta(G) \geq k+2$, we deduce that
\begin{equation}\label{31}
|X'|+|Y'|+|T|\geq d_{G}(u) \geq k+2.
\end{equation}
Combining with $|X_{e}^{u}|+|X_{e'}^{u}|=2k+2$, we obtain that
\begin{equation}\label{32}
|X''|+|Y''|+|T| \leq k.
\end{equation}
Furthermore, by (\ref{32}), $|Y'|+|Y''|+|T|=k+1>|X''|+|Y''|+|T|$,
which implies that
\begin{equation}\label{032}
|X''|<|Y'|.
\end{equation}

\begin{claim}\label{3}
$I_{2} \neq \emptyset$.
\end{claim}

\noindent Proof of Claim \ref{3}.
Suppose to the contrary that $I_{2}=\emptyset$. Then $|V(G_{4})|=|I_{1} \cup X''|$.
So $G[V(G_{4})\cup \{y\}]=G[I_{1} \cup X'' \cup \{y\}]$.
Since $G[V(G_{4})\cup \{y\}]$ is an odd component of $G_{e'}$,
$G[I_{1} \cup X'' \cup \{y\}]$ is connected. Then for the edge $uy$,
there exists a $(k+1)$-vertex cut $(X_{e'}^{u} \backslash \{y\}) \cup \{u\}$
such that $G[I_{1} \cup X'' \cup \{y\}]$ is an odd component
of $G-((X_{e'}^{u} \backslash \{y\}) \cup \{u\})$ containing $y$.
Moreover, $u$ has unique neighbor $y$ in $G[I_{1} \cup X'' \cup \{y\}]$.
So $(X_{e'}^{u} \backslash \{y\}) \cup \{u\}$ satisfies Property $Q$.
By (\ref{032}),
$|V(G[I_{1} \cup X'' \cup \{y\}])| \leq |V(G[I_{1} \cup Y'])|<|V(G_{1})|$,
contradicting the minimality of $G_{1}$.
\hfill $\square$

\smallskip
By Claim \ref{3}, $X'' \cup Y'' \cup T$ is a vertex cut.
Since $G$ is $k$-connected, $|X''|+|Y''|+|T|\geq k$.
By (\ref{32}), we have
\begin{equation}\label{13}
|X''|+|Y''|+|T|=k.
\end{equation}
Then
\begin{equation}\label{14}
|X'|+|Y'|+|T|=k+2.
\end{equation}
Combining (\ref{30}), (\ref{14}) and $d_{G}(u)\geq k+2$,
we have the following claim.

\begin{claim}\label{40}
$N(u)=X' \cup Y' \cup T$.
\end{claim}

\begin{claim}\label{4}
$X'' \neq \emptyset$.
\end{claim}

\noindent Proof of Claim \ref{4}. Suppose to the contrary that $X''=\emptyset$.
Then $G[V(G_{4}) \cup \{y\}]=G[I_{1} \cup I_{2} \cup \{y\}]$.
Since $e'$ is of type 1 and $G[V(G_{4}) \cup \{y\}]$ is an odd
component of $G_{e'}$, $G[V(G_{4}) \cup \{y\}]$ is connected.
But $G[I_{1} \cup I_{2} \cup \{y\}]$ is disconnected because
there is no path connecting
$I_{1} \cup \{y\}$ and $I_{2}$ through a vertex in $X''$, a contradiction.
\hfill $\square$

\smallskip
Suppose that $|X''|+|Y'|+|T|\leq k$.
If $|X''|+|Y'|+|T|<k$, then $I_{1}=\emptyset$ as $G$ is $k$-connected.
Thus $N(y) \subseteq X'' \cup (Y' \backslash \{y\}) \cup T \cup \{u\}$.
So $d_{G}(y)<k$, a contradiction to the assumption that $\delta(G)\geq k+2$.

If $|X''|+|Y'|+|T|=k$, then $|I_{1}|$ is even
as $G-(X'' \cup Y' \cup T)$ has a perfect matching.
Thus $G-(X'' \cup (Y' \backslash \{y\}) \cup T \cup \{u\})$
has a subgraph of $G[I_{1} \cup \{y\}]$ with odd order.
So $G-(X'' \cup (Y'\backslash \{y\}) \cup T \cup \{u\})$
has no perfect matching, contradicting that $G$ is $k$-factor-critical.

It follows that
\begin{equation}\label{33}
|X''|+|Y'|+|T| \geq k+1.
\end{equation}
Together with $|X_{e}^{u}|+|X_{e'}^{u}|=2k+2$ again, we deduce that
\begin{equation}\label{34}
|X'|+|Y''|+|T| \leq k+1.
\end{equation}

Next we will prove that (\ref{33}) and (\ref{34}) are true only if
the equality holds in both (\ref{33}) and (\ref{34})
by considering whether $v \in X'$.

If $v\in X'$, then by arguments similar to
those that led to (\ref{33}), we have $|X'|+|Y''|+|T| \geq k+1$.
Combining with (\ref{34}),
we obtain that $|X'|+|Y''|+|T|=k+1$. Then $|X''|+|Y'|+|T|=k+1$.

If $v\notin X'$, then $v \in T$ and further $v \in X_{e'}^{u}$.
By Claim \ref{3} and (\ref{13}), $X'' \cup Y'' \cup T$ is a $k$-vertex cut.
If $N(v) \cap I_2=\emptyset$,
then $X'' \cup Y'' \cup (T\backslash \{v\})$ is a $(k-1)$-vertex cut,
contradicting that $G$ is $k$-connected.
So $N(v) \cap I_{2} \neq \emptyset$ (as shown in Fig. \ref{tu-05}).

\begin{figure}[h]
\centering
\includegraphics[height=5cm,width=5cm]{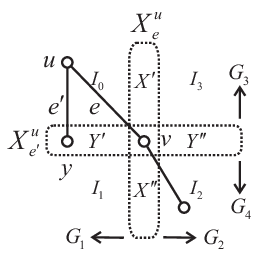}
\caption{\label{tu-05} $N(v) \cap I_{2} \neq \emptyset$ when $v\notin X'$.}
\end{figure}

We show that $|X'|+|Y''|+|T|=k+1$.
It suffices to show that $|X'|+|Y''|+|T|\geq k+1$ by (\ref{34}).

Suppose to the contrary that $|X'|+|Y''|+|T| \leq k$.
We claim that $I_{3}=\emptyset$. Otherwise,
$X' \cup Y'' \cup T$ would be a $k$-vertex cut as $G$ is $k$-connected,
that is, $|X'|+|Y''|+|T|=k$. So $N(v) \cap I_{3} \neq \emptyset$.
Let $v_2\in N(v) \cap I_{2}$ and $v_3\in N(v) \cap I_{3}$.
Then $G[\{v, v_2, v_3, u\}]$ is a claw, a contradiction.

Since $|X'|+|Y''|+|T|\leq k$ and $|Y'|+|Y''|+|T|=k+1$, $|X'|<|Y'|$.
Noticing that $G_{1}$ is of the minimum order,
$|V(G_{1})| \leq |V(G_{3})|$.
So $|I_{0} \cup I_{1} \cup Y'| \leq |I_{0} \cup I_{3} \cup X'|=|I_{0} \cup X'|$.
Then $|Y'|\leq |I_{1}|+|Y'| \leq |X'|<|Y'|$, a contradiction.

Thus $|X'|+|Y''|+|T| \geq k+1$. By (\ref{34}), we obtain that
\begin{equation}\label{35}
|X'|+|Y''|+|T|=k+1
\end{equation}
and then
\begin{equation}\label{36}
|X''|+|Y'|+|T|=k+1.
\end{equation}
Since $|Y'|+|Y''|+|T|=k+1$, by (\ref{35}), we have
\begin{equation}\label{360}
|X'|=|Y'|.
\end{equation}

\begin{claim}\label{5}
$N(y) \cap I_{1}\neq \emptyset$.
\end{claim}

\noindent Proof of Claim \ref{5}.
Suppose to the contrary that $N(y) \cap I_{1}=\emptyset$,
then $N(y) \subseteq X'' \cup (Y' \backslash \{y\}) \cup T \cup \{u\}$.
By (\ref{36}), $|X'' \cup (Y' \backslash \{y\})\cup T \cup \{u\}|=k+1$.
So $d_{G}(y) \leq k+1$, contradicting that $\delta(G)\geq k+2$.
\hfill $\square$

\begin{claim}\label{6}
$G[I_{1}\cup \{y\}]$ is connected.
\end{claim}

\noindent Proof of Claim \ref{6}.
Suppose to the contrary that $G[I_{1} \cup \{y\}]$ is disconnected.
Let $C_{1}$ be the component of $G[I_{1} \cup \{y\}]$ containing $y$
and $C_{2}=G[(I_{1} \cup \{y\})\backslash V(C_{1})]$.
By (\ref{36}), $|X'' \cup (Y' \backslash \{y\}) \cup T|=k$.
Then $X'' \cup (Y' \backslash \{y\}) \cup T$ is a $k$-vertex cut.
Thus every vertex in $X'' \cup (Y' \backslash \{y\}) \cup T$
has a neighbor in every component of $C_{2}$.
By Claim \ref{5}, $C_{1}$ is nontrivial. We claim that
there exists a vertex $x$ in $X'' \cup (Y' \backslash \{y\}) \cup T$
such that $N(x)$ $\cap$ $(V(C_{1}) \backslash \{y\})\neq \emptyset$.
Otherwise, $\{y\}$ is a vertex cut,
contradicting that $G$ is $k$-connected $(k\geq 2)$.
So $x$ has neighbours in both $V(C_{1}) \backslash \{y\}$ and
each component of $C_2$.

If $x \in X''$, then $N(x) \cap I_{2} \neq \emptyset$ as
$X'' \cup Y'' \cup T$ is a $k$-vertex cut.
If $x \in (Y'\backslash \{y\}) \cup T$, then $ux \in E(G)$ by Claim \ref{40}.
In either case, $G$ has a claw, a contradiction.
\hfill $\square$

\begin{claim}\label{600}
$|I_{1}|$ is odd.
\end{claim}

\noindent Proof of Claim \ref{600}.
Suppose to the contrary that $|I_{1}|$ is even. By Claim \ref{6},
$G[I_{1}\cup \{y\}]$ is a connected graph of odd order.
By (\ref{36}), $|X'' \cup (Y' \backslash \{y\}) \cup T \cup \{u\}|=k+1$.
Since $E(I_{1}\cup \{y\}, \{u\})=\{uy\}$,
$X'' \cup (Y' \backslash \{y\}) \cup T \cup \{u\}$ satisfies Property $Q$.
But $|V(G[I_{1}\cup \{y\}])|<|V(G_1)|$,
a contradiction to the minimality of $G_{1}$.
\hfill $\square$

\begin{claim}\label{7}
$T=\emptyset$.
\end{claim}

\noindent Proof of Claim \ref{7}. Suppose to the contrary that $T \neq \emptyset$.
By Claim \ref{40}, every vertex in $X' \cup Y' \cup T$ is adjacent to $u$.
By Claim \ref{3} and (\ref{13}),
every vertex in $X'' \cup Y'' \cup T$ has a neighbor in $I_{2}$.

Next we will prove that every vertex in $X'' \cup Y' \cup T$ has a neighbor in $I_{1}$.

By Claim \ref{600}, $|I_{1}|$ is odd.
If there is a vertex $x$ in $X'' \cup Y' \cup T$
such that $N(x) \cap I_{1}=\emptyset$.
Then $G-((X'' \cup Y' \cup T) \backslash \{x\})$
has a subgraph of $G[I_{1}]$ with odd order,
yielding $G-((X'' \cup Y' \cup T)\backslash \{x\})$ has no perfect matching.
This contradicts that $G$ is $k$-factor-critical
as $|(X'' \cup Y' \cup T)\backslash \{x\}|=k$.
Then every vertex in $X'' \cup Y' \cup T$ has a neighbor in $I_{1}$.

In a word, every vertex in $T$ has a neighbor in $\{u\}$
(resp. $I_{1}$ and $I_{2}$), which implies that $G$ contains a claw,
a contradiction.
\hfill $\square$

\begin{claim}\label{8}
$|N(u) \cap V(G_{1})| \geq 2$.
\end{claim}

\noindent Proof of Claim \ref{8}.
By Claims \ref{40} and \ref{7}, $N(u)=X' \cup Y' \cup T=X' \cup Y'$.
If $y$ is the only neighbor of $u$ in $V(G_{1})$,
then $N(u) \subseteq \{y\} \cup X_{e}^{u}$.
Since $d_{G}(u)\geq k+2$ and $|X_{e}^{u}|=k+1$, $N(u)=\{y\} \cup X_{e}^{u}$.
Then $N(u)=X' \cup Y'$=$\{y\} \cup X_{e}^{u}$.
Thus $X'=X_{e}^{u}$ and $Y'=\{y\}$.
Consequently, $X''=\emptyset$, which contradicts Claim \ref{4}.
\hfill $\square$

\smallskip
By Claim \ref{8}, let $z \in N(u) \cap (V(G_{1})\backslash \{y\})$.

Finally, we consider the edge $uz$. Let $e''=uz$. By Propositions \ref{prop1},
$e''$ is either of type 1 or of type 2.
Since $v\in N(u)$ and $v\notin N(z)$, $N(u) \backslash \{z\} \neq N(z)\backslash \{u\}$.
By Propositions \ref{prop0}, $e''$ is of type 1.
Then there exists $S_{e''} \subseteq V(G)\backslash \{u, z\}$ with $|S_{e''}|=k$
such that $G-e''-S_{e''}$ has exactly two odd components containing
$u$ and $z$, respectively.
Let $X_{e''}^{u}=S_{e''} \cup \{z\}$. By Lemma \ref{lem00},
$X_{e''}^{u}$ satisfies Property $Q$ (as shown in Fig. \ref{tu-06} (a)).
Let $G_{5}$ be the odd component of $G-X_{e''}^{u}$ containing $u$
and $G_{6}=G-X_{e''}^{u}-V(G_{5})$.
Since $G[V(G_6) \cup \{z\}]$ is an odd component of $G-e''-S_{e''}$, $|V(G_6)|$ is even.

s
\begin{figure}[h]
\centering
\includegraphics[height=5cm,width=10cm]{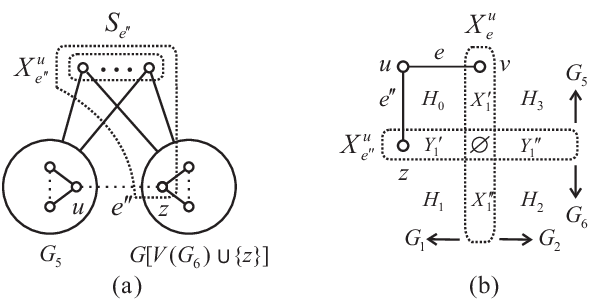}
\caption{\label{tu-06} (a) $X_{e''}^{u}$ satisfies Property $Q$;
(b) $X_{e}^{u}$ and $X_{e''}^{u}$ separate the graph $G$.}
\end{figure}

Again, we view the graph $G$ by $X_{e}^{u}$ and $X_{e''}^{u}$.
Let $T_1=X_{e}^{u}\cap X_{e''}^{u}$,
$X_{1}'=X_{e}^{u} \cap V(G_{5})$, $X_{1}''=X_{e}^{u} \cap V(G_{6})$,
$Y_{1}'=X_{e''}^{u} \cap V(G_{1})$ and $Y_{1}''=X_{e''}^{u}\cap V(G_{2})$.
And let $H_{0}=V(G_{1})$ $\cap$ $V(G_{5})$, $H_{1}=V(G_{1})$ $\cap$ $V(G_{6})$,
$H_{2}=V(G_{2}) \cap V(G_{6})$ and $H_{3}=V(G_{2}) \cap V(G_{5})$
(as shown in Fig. \ref{tu-06} (b)).
By arguments similar to $X_{e}^{u}$ and $X_{e'}^{u}$,
we obtain that $T_1=\emptyset$ and the other similar conclusions as follows:

\begin{equation}\label{21}
|X_{1}''|+|Y_{1}''|=k
\end{equation}
and
\begin{equation}\label{22}
|X_{1}'|+|Y_{1}''|=k+1.
\end{equation}
Together with $|X_{e}^{u}|+|X_{e''}^{u}|=2k+2$, we have
\begin{equation}\label{23}
|X_{1}'|+|Y_{1}'|=k+2
\end{equation}
and
\begin{equation}\label{24}
|X_{1}''|+|Y_{1}'|=k+1.
\end{equation}
Then we also have $N(u)=X_{1}' \cup Y_{1}'$.
Moreover, $G[V(H_1)\cup \{z\}]$ is connected and $|H_1|$ is odd (call this conclusion $(\ddag)$).

\begin{claim}\label{9}
$X_{1}'=X'$, $X_{1}''=X''$, $Y_{1}'=Y'$, $H_0=I_0$ and $H_1=I_1$.
\end{claim}

\noindent Proof of Claim \ref{9}.
By Claims \ref{40} and \ref{7}, $N(u)=X' \cup Y'$.
Since $N(v) \cap V(G_{1})=\{u\}$ and $N(y) \cap V(G_{3})=\{u\}$,
$N(v) \cap Y'=\emptyset$ and $N(y) \cap X'=\emptyset$.
Because $G$ is claw-free and $vy\notin E(G)$,
both $G[X']$ and $G[Y']$ are complete graphs.
Likewise, both $G[X_{1}']$ and $G[Y_{1}']$ are complete graphs.
Since $v \in X'$ and $v \in X_{1}'$, $X_{1}'=X'$.
Then, by $N(u)=X_{1}' \cup Y_{1}'=X' \cup Y'$, we obtain that $Y_{1}'=Y'$.
So $X_{1}''=X_{e}^{u} \backslash X_{1}'=X_{e}^{u} \backslash X'=X''$.

We are left to show that $H_0=I_0$ and $H_1=I_1$.
Note that $V(G_1)=I_0 \cup I_1 \cup Y'=H_0 \cup H_1 \cup Y_{1}'$.
It suffices to show that $H_0=I_0$ by $Y_{1}'=Y'$.

We claim that $X' \backslash\{v\} \neq \emptyset$.
Otherwise, $|X'|=1$, $|X''|=k$ and $|Y'|=k+2-|X'|=k+1$.
Hence $|X''|+|Y'|=k+k+1=2k+1$, a contradiction to (\ref{36}).

Then $|X_{1}'|=|X'|\geq 2$. By (\ref{23}), $|Y_{1}'|=k+2-|X_{1}'|\leq k$.

If $I_{0}\backslash\{u\}=\emptyset$,
then we will show that $H_{0}\backslash\{u\}=\emptyset$ too.
If so, then $H_0=I_0$. Suppose to the contrary that $H_{0}\backslash\{u\} \neq \emptyset$.
Then $H_{0}\backslash\{u\}\subseteq I_1$.
Since there is no edge joining $X'$ and $I_1$, every vertex in $X_{1}'$ has
no neighbor in $H_{0}\backslash\{u\}$.
However, $G$ is $k$-connected and $N(z) \cap (H_{0}\backslash\{u\})=\emptyset$.
Then every component of $G[H_{0}\backslash\{u\}]$ is connected to
at least $k$ vertices in $Y_{1}'\backslash \{z\}$.
Thus $|Y_{1}'\backslash \{z\}|\geq k$. So $|Y_{1}'|\geq k+1$,
contradicting that $|Y_{1}'|\leq k$.

If $I_{0}\backslash\{u\} \neq \emptyset$ and $H_{0}\backslash\{u\}=\emptyset$,
then $I_{0}\backslash\{u\}\subseteq H_1$.
By symmetry, we can give a similar discussion to the previous paragraph.

Suppose that $I_{0}\backslash\{u\} \neq \emptyset$ and $H_{0}\backslash\{u\}\neq \emptyset$.
Since $u, v$ and $y$ have no neighbor in $I_{0}\backslash \{u\}$,
$(X'\backslash \{v\})\cup (Y'\backslash \{y\})$ is a $k$-vertex cut by (\ref{14}) and Claim \ref{7}.
Then every vertex in $X'\backslash \{v\}$ has a neighbor in
every component of $G[I_{0}\backslash\{u\}]$.
Similarly, because $u, v$ and $z$ have no neighbor in $H_{0}\backslash \{u\}$,
$(X_{1}'\backslash \{v\})\cup (Y_{1}'\backslash \{y\})$ is a $k$-vertex cut by (\ref{23}).
Then every vertex in $X_{1}'\backslash \{v\})$ has a neighbor in
every component of $G[H_{0}\backslash\{u\}]$.
Combining $X'=X_{1}'$ and $Y'=Y_{1}'$, we have $I_0 \backslash \{u\}=H_0 \backslash \{u\}$.
So $I_{0}=H_0$.
\hfill $\square$

\begin{figure}[h]
\centering
\includegraphics[height=4.5cm,width=4.5cm]{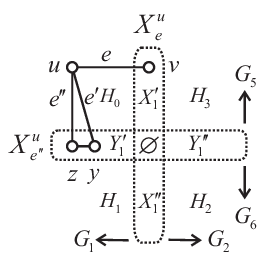}
\caption{\label{tu-07} $u$ is the only neighbor of $y$ and $z$ in $H_0\cup X_{1}'$.}
\end{figure}

Since both $uy$ and $uz$ are of type 1, $N(y) \cap (I_0 \cup X')=\{u\}$
and $N(z) \cap (H_0 \cup X_{1}')=\{u\}$. By $I_0=H_0$ and $X'=X_{1}'$,
we obtain that $N(y) \cap (H_0 \cup X_{1}')=N(z) \cap (H_0 \cup X_{1}')=\{u\}$
(as shown in Fig. \ref{tu-07}).

By the conclusion $(\ddag)$, $G[H_{1} \cup \{y, z\}]$ is an odd component
of $G-(X_{1}'' \cup (Y_{1}' \backslash\{y, z\}) \cup \{u\})$, yielding
$G-(X_{1}'' \cup (Y_{1}' \backslash\{y, z\}) \cup \{u\})$ has no perfect matching.
This contradicts that $G$ is $k$-factor-critical
as $|X_{1}'' \cup (Y_{1}' \backslash\{y, z\}) \cup \{u\}|=k$.

\smallskip
Then we have reached the final contradiction, which means that
our assumption that $\delta(G) \geq k+2$ must be false.
Therefore, $\delta(G)=k+1$. So the result follows.
\hfill $\square$

\section{Proof of Theorem \ref{main2}}

In this section, we prove that there are even many vertices of
degree $k+1$ in the minimal $k$-factor-critical claw-free graph
if it is $(k+1)$-connected.

\noindent{\bf Proof of Theorem \ref{main2}.}
By Theorem \ref{main1}, $\delta(G)=k+1$.
Denote by $V_{1}$ the set of vertices of degree $k+1$ in $G$.
Let $V_2=V(G)\backslash V_{1}$, and $G_{1} :=G[V_{1}]$ and $G_{2} :=G[V_{2}]$.
Then for every vertex $x$ of $G_{2}$, $d_{G}(x)\geq k+2$.
By Proposition \ref{prop1}, for every edge $e$ of $G_2$,
$e$ is either of type 1 or of type 2.
Let $E_1$ be the set of edges of type 1 in $G_2$ and $E_2=E(G_2)\backslash E_1$.
Then $E_2$ consists of edges of type 2.
We first prove the following claim.

\bigskip
\noindent \textbf{Claim 1.}
$G[E_1]$ is a forest.
\bigskip

Suppose to the contrary that $G[E_1]$ contains a cycle
$C=u_{1}u_{2}\cdots u_{t}u_{1}$ ($t \geq 3$).
Since for every edge $e=u_iu_j \in E(C)$, $e$ is of type 1,
$G$ has the configuration as shown in Fig. \ref{tu-02} (a) by Lemma \ref{lem1}.
Then there exists two $(k+1)$-vertex cuts $S_e\cup \{u_j\}$
and $S_e\cup \{u_i\}$ in $G$.
Let $X$ be the set of all such $(k+1)$-vertex cuts relative to
the edges of type 1 on $C$.

Among all these $(k+1)$-vertex cuts in $X$,
we choose one such that the removal of it
results in an odd component of the minimum order.

Let $X_{u_1u_2}^{u_2}$ be the $(k+1)$-vertex cut. And let
$H_{u_{2}u_{1}}$ be the odd component of $G-X_{u_1u_2}^{u_2}$ containing $u_{2}$
and $H_{u_{2}u_{1}}$ be of the minimum order.
Let $H_{1}'=G-X_{u_1u_2}^{u_2}-V(H_{u_{2}u_{1}})$ (as shown in Fig. \ref{tu-08} (a)).
Since $G[V(H_{1}')\cup \{u_1\}]$ is an odd component of $G-u_1u_2-S_{u_1u_2}$,
$|V(H_{1}')|$ is even.

\begin{figure}[h]
\centering
\includegraphics[height=5cm,width=9cm]{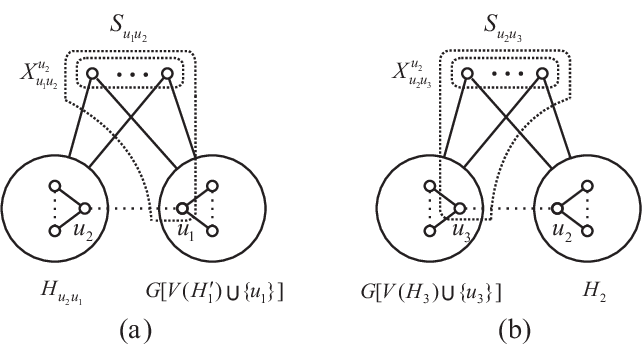}
\caption{\label{tu-08}
  (a) $X_{u_1u_2}^{u_2}$ is a $(k+1)$-vertex cut;
  (b) $X_{u_2u_3}^{u_2}$ is a $(k+1)$-vertex cut.}
\end{figure}

We consider another edge $u_{2}u_{3}$ on $C$. Let $X_{u_2u_3}^{u_2}=S_{u_2u_3}\cup \{u_3\}$.
Then $X_{u_2u_3}^{u_2}$ is a $(k+1)$-vertex cut of $G$.
Denote by $H_{2}$ the odd component of $G-X_{u_2u_3}^{u_2}$ containing $u_{2}$ and
let $H_{3}=G-X_{u_2u_3}^{u_2}-V(H_{2})$ (as shown in Fig. \ref{tu-08} (b)).
Because $G[V(H_{3})\cup \{u_3\}]$ is an odd component of $G-u_2u_3-S_{u_2u_3}$,
$|V(H_{3})|$ is even.

We claim that $X_{u_1u_2}^{u_2} \neq X_{u_2u_3}^{u_2}$. Otherwise,
$u_{3}\in X_{u_1u_2}^{u_2}$, $H_{u_{2}u_{1}}=H_{2}$ and $H_{1}'=H_{3}$.
Since $u_{1}$ has the only neighbor $u_{2}$ in $H_{u_{2}u_{1}}$
and $u_{3}$ has the only neighbor $u_{2}$ in $H_2$,
$u_{2}$ is the only neighbor of $u_{1}$ and $u_{3}$ in $H_{u_{2}u_{1}}$.
Then $(X_{u_1u_2}^{u_2} \backslash \{u_{1}, u_{3}\})\cup \{u_{2}\}$ is a
$k$-vertex cut, contradicting that $G$ is $(k+1)$-connected.

Next we present some notation below.

Let $T_{1}=X_{u_1u_2}^{u_2} \cap X_{u_2u_3}^{u_2}$,
$X_{1}'=X_{u_1u_2}^{u_2} \cap V(H_{2})$, $X_{1}''=X_{u_1u_2}^{u_2} \cap V(H_{3})$,
$Y_{1}'=X_{u_2u_3}^{u_2} \cap V(H_{u_{2}u_{1}})$ and $Y_{1}''=X_{u_2u_3}^{u_2} \cap V(H_{1}')$.
Then $X_{u_1u_2}^{u_2}=X_{1}'\cup X_{1}''\cup T_{1}$
and $X_{u_2u_3}^{u_2}=Y_{1}'\cup Y_{1}''\cup T_{1}$.
Let $J_{0}=V(H_{u_{2}u_{1}})$ $\cap$ $V(H_{2})$,
$J_{1}=V(H_{u_{2}u_{1}})\cap V(H_{3})$,
$J_{2}=V(H_{1}')\cap V(H_{3})$ and $J_{3}=V(H_{1}')\cap V(H_{2})$
(as shown in Fig. \ref{tu-09}). Since $u_1, u_3 \in N(u_2)$,
$u_{1}\in X_{1}'\cup T_{1}$ and $u_{3}\in Y_{1}'\cup T_{1}$.

\begin{figure}[h]
\centering
\includegraphics[height=4.5cm,width=4.5cm]{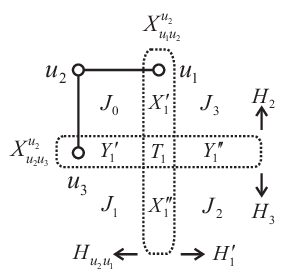}
\caption{\label{tu-09} $X_{u_1u_2}^{u_2}$
and $X_{u_2u_3}^{u_2}$ separate the graph $G$.}
\end{figure}

We will prove that $|X_{1}'|+|Y_{1}'|+|T_{1}|\geq k+2$.
Since $G$ is $(k+1)$-connected and $J_0\neq \emptyset$,
$|X_{1}'|+|Y_{1}'|+|T_{1}|\geq k+1$.
Suppose to the contrary that $|X_{1}'|+|Y_{1}'|+|T_{1}|=k+1$.
Since $N(u_{2}) \subseteq X_{1}'\cup Y_{1}'\cup T_{1}\cup J_0$
and $d_{G}(u_2)\geq k+2$, $N(u_2)\cap J_0 \neq \emptyset$.
Then $J_0 \backslash \{u_2\}\neq \emptyset$.
Noticing that $u_1$ has the only neighbor $u_2$ in $H_{u_{2}u_{1}}$
and $u_3$ has the only neighbor $u_2$ in $H_{2}$,
$u_2$ is the only neighbor of $u_1$ and $u_3$ in $J_0$.
Then $((X_{1}'\cup Y_{1}'\cup T_{1})\backslash \{u_1, u_3\}) \cup \{u_2\}$
is a $k$-vertex cut, contradicting that $G$ is $(k+1)$-connected.
Thus
\begin{equation}\label{45}
|X_{1}'|+|Y_{1}'|+|T_{1}|\geq k+2.
\end{equation}
Together with $|X_{u_1u_2}^{u_2}|+|X_{u_2u_3}^{u_2}|=2k+2$, we have
\begin{equation}\label{46}
|X_{1}''|+|Y_{1}''|+|T_{1}|\leq k.
\end{equation}
Since $G$ is $(k+1)$-connected, $J_{2}=\emptyset$.
Consequently, $J_{1}\cup X_{1}''=V(H_{3})$.
Then $|J_{1}\cup X_{1}''|$ is even as $|V(H_{3})|$ is even.

Let $H_{u_{3}u_{2}}=G[J_{1}\cup X_{1}''\cup \{u_{3}\}]$ and
$X_{u_2u_3}^{u_3}=(X_{u_2u_3}^{u_2} \backslash \{u_{3}\})\cup \{u_{2}\}$.
Then $X_{u_2u_3}^{u_3}$ is a $(k+1)$-vertex cut
relative the edge $u_2u_3$ of type 1 (as shown in Fig. \ref{tu-10}).
Since $G[J_{1} \cup X_{1}''$ $\cup$ $\{u_{3}\}]=G[V(H_{3}) \cup \{u_3\}]$
and $G[V(H_{3})\cup \{u_3\}]$ is an odd component of $G-u_2u_3-S_{u_2u_3}$,
$H_{u_{3}u_{2}}$ is connected.
Note that $G$ has the configuration corresponding to the edge $u_2u_3$ of type 1.

\begin{figure}[h]
\centering
\includegraphics[height=4.3cm,width=3.7cm]{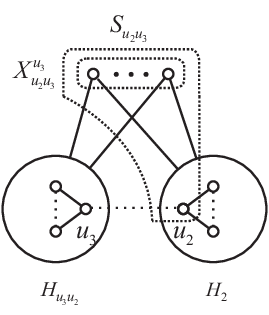}
\caption{\label{tu-10} $X_{u_2u_3}^{u_3}$ is a $(k+1)$-vertex cut.}
\end{figure}

By (\ref{46}),
$|X_{1}''|+|Y_{1}''|+|T_{1}|\leq k<|Y_{1}'|+|Y_{1}''|+|T_{1}|=k+1$.
Then $|X_{1}''|<|Y_{1}'|$. It follows that
$|J_{1}\cup X_{1}''\cup \{u_{3}\}|\leq|J_{1}\cup Y_{1}'|<|V(H_{u_{2}u_{1}})|$,
that is, $|V(H_{u_{3}u_{2}})|<|V(H_{u_{2}u_{1}})|$,
a contradiction to the minimality of $H_{u_{2}u_{1}}$.
So Claim 1 holds.

\bigskip
By Claim 1, we have
\begin{equation}\label{17}
|E_1|\leq |V(G[E_1])|-1.
\end{equation}
By Proposition \ref{prop2}, every vertex in $G_2$ is incident with at most
one edge of type 2. Then $E_2$ is a matching  of $G_2$. Thus
\begin{equation}\label{18}
2|E_2|=|V(G[E_2])|.
\end{equation}
Moreover, the sum of degrees of all vertices in $G_{2}$ is $2|E(G_2)|$.
By taking (\ref{17}) and (\ref{18}) into account, we have
\begin{center}
$2|E(G_2)|=2|E_1|+2|E_2|\leq 2(|V(G[E_1])|-1)+|V(G[E_2])|\leq 2(|V_2|-1)+|V_2|=3|V_2|-2$.
\end{center}


We can estimate the number of edges between $V_1$ and $V_2$ as follows.
\begin{center}
$(k+1)|V_1|\geq |E(V_1, V_2)|\geq (k+2)|V_2|-2|E(G_2)| \geq (k+2)|V_2|-(3|V_2|-2)$.
\end{center}
Then $(k+1)|V_1|\geq (k-1)|V_2|+2$. So $2k|V_1|\geq (k-1)(|V_1|+|V_2|)+2>(k-1)(|V_1|+|V_2|)$.
It follows that
\begin{center}
$\dfrac{|V_1|}{|V_1|+|V_2|}>\dfrac{k-1}{2k}$.
\end{center}

Therefore, $G$ has at least $\frac{k-1}{2k}|V(G)|$
vertices of degree $k+1$.
\hfill $\square$

\bigskip

If $k=2$, we can obtain the following corollary by Theorem \ref{main2},
which yields further evidence for Conjecture \ref{conj2}.

\begin{cor}\label{cubic}
Every $3$-connected minimal bicritical claw-free graph $G$ has at least $\frac{1}{4}|V(G)|$
vertices of degree three.
\end{cor}


\begin{thebibliography}{99}
\small \setlength{\itemsep}{-0.2mm}

\bibitem{BM} J. A. Bondy and U. S. R. Murty, Graph Theory, New York: Springer, 2008.

\bibitem{CLM} M. H. de Carvalho, C. L. Lucchesi and U. S. R. Murty,
How to build a brick, Discrete Math. 306 (2006) 2386-2410.

\bibitem{ELW} J. Edmonds, L. Lov\'{a}sz and W. R. Pulleyblank,
Brick decompositions and the matching rank of graphs, Combinatorica 2 (1982) 247-274.

\bibitem{F} O. Favaron, On $k$-factor-critical graphs,
Discuss. Math. Graph Theory 16 (1996) 41-51.

\bibitem{FS} O. Favaron and M. Shi, Minimally $k$-factor-critical graphs,
Australas. J. Combin. 17 (1998) 89-97.

\bibitem{GT} T. Gallai, Neuer Beweis eines Tutte'schen Satzes, Magyar Tud. Akad. Mat.
 Kutat\'{o} Int. K\"{o}zl. 8 (1963) 135-139.

\bibitem{GZ1} J. Guo and H. Zhang, Minimally $k$-factor-critical graphs for some large $k$,
Graphs and Combin. 39 (2023) 60.

\bibitem{GZ2} J. Guo and H. Zhang, Minimum degree of minimal $(n-10)$-factor-critical graphs,
Discrete Math. 347 (2024) 113839.

\bibitem{GZ3} J. Guo, H. Wu and H. Zhang, Cubic vertices of minimal bicritical graphs,
Discrete Appl. Math. 352 (2024) 44-48.

\bibitem{HL} X. He and F. Lu, The cubic vertices of solid minimal bricks,
Discrete Math. 347 (2024) 113746.

\bibitem{LLZH} Q. Li, F. Lu and H. Zhang, On minimal $k$-factor-critical planar graphs,
submitted.

\bibitem{LZL} F. Lin, L. Zhang and F. Lu, The cubic vertices of minimal bricks,
J. Graph Theory 76 (2014) 20-33.

\bibitem{LR} C. H. C Little and F. Rendl, Operations preserving
the Pfaffian property of a graph, J. Aust. Math. Soc. A 50 (1991) 248-275.

\bibitem{LDYQ} D. Lou and Q. Yu, Sufficient conditions for $n$-matchable graphs,
Australas. J. Combin. 29 (2004) 127-133.

\bibitem{LL} L. Lov\'{a}sz, On the structure of factorizable graphs, Acta Math. Acad. Sci.
Hungar. 23 (1972) 179-195.

\bibitem{LO} L. Lov\'{a}sz, Matching structure and the matching lattice,
J. Combin. Theory Ser. B 43 (1987) 187-222.

\bibitem{LP} L. Lov\'{a}sz and M. D. Plummer, Matching Theory, Ann. Discrete Math., Vol. 29,
North-Holland, Amsterdam, 1986; AMS Chelsea Publishing, Amer. Math. Soc., Providence, 2009.

\bibitem{MD} W. Mader, Eine Eigenschaft der Atome endlicher Graphen,
Arch. Math. 22 (1971) 333-336.

\bibitem{MW} W. Mader, Zur Struktur minimal $n$-fach zusammenhangende Graphen,
Abh. Math. Sem. Universit\"{a}t Hamburg 49 (1979) 49-69.

\bibitem{N} T. Nishimura, A closure concept in factor-critical graphs, Discrete Math.
259 (2002) 319-324.

\bibitem{NT} S. Norine and R. Thomas, Minimal bricks,
J. Combin. Theory Ser. B 96 (2006) 505-513.

\bibitem{PMD} M. D. Plummer, Degree sums, neighborhood unions and matching extension in graphs.
In: R. Bodendiek, ed., Contemporary Methods in Graph Theory, B. I.
Wissenschaftsverlag, Mannheim, 1990, pp. 489-502.

\bibitem{PS} M. D. Plummer and A. Saito, Closure and factor-critical graphs,
Discrete Math. 215 (2000) 171-179.

\bibitem{TTW} W. T. Tutte, The factorization of linear graphs,
J. Lond. Math. Soc. 22 (1947) 107-111.

\bibitem{Y} Q. Yu, Characterizations of various matching extensions in graphs,
Australas. J. Combin. 7 (1993) 55-64.

\bibitem{YL} Q. Yu and G. Liu, Graph Factors and Matching Extensions,
Higher Education Press, Beijing, 2009.

\bibitem{ZWL}Z. Zhang, T. Wang and D. Lou, Equivalence between extendibility
and factor-criticality, Ars Combin. 85 (2007) 279-285.

\bibitem{ZWY22}Y. Zhang, X. Wang and J. Yuan, Bicritical graphs without removable edges, Discrete Appl.
Math. 320 (2022) 1-10.



\end{thebibliography}
\end{document}